\documentclass[11pt]{amsart}

\usepackage{amssymb,amsmath}
\usepackage{bbm}
\usepackage{a4wide}

\newtheorem{theorem}{Theorem}
\newtheorem{prop}{Proposition}
\newtheorem{lemma}{Lemma}
\newtheorem{corollary}{Corollary}

\newcommand{\ts}{\hspace{0.5pt}}

\newcommand{\RR}{\mathbb{R}\ts}
\newcommand{\QQ}{\mathbb{Q}\ts}
\newcommand{\ZZ}{\mathbb{Z}}
\newcommand{\TT}{\mathbb{T}}
\newcommand{\NN}{\mathbb{N}}
\newcommand{\sss}{\mathbb{S}}

\newcommand{\one}{\mathbbm{1}}
\newcommand{\ep}{\text{\raisebox{1.5pt}{$\scriptstyle \ts\bigwedge\!$}}}

\DeclareMathOperator{\diag}{diag}
\DeclareMathOperator{\sgn}{sgn}
\DeclareMathOperator{\trace}{tr}

\DeclareMathOperator{\fix}{Fix}
\DeclareMathOperator{\Mat}{Mat}
\DeclareMathOperator{\Lie}{Lie}

\begin{document}

\vspace*{-5mm}

\title[Dynamical zeta function]
{A note on the dynamical zeta function \\[2mm]
of general toral endomorphisms}

\author{Michael Baake}
\author{Eike Lau}
\author{Vytautas Paskunas}
\address{Fakult\"at f\"ur Mathematik, Universit\"at Bielefeld, \newline
\hspace*{\parindent}Postfach 100131, 33501 Bielefeld, Germany}
\email{$\{$mbaake,lau,paskunas$\}$@math.uni-bielefeld.de}

\begin{abstract} 
  It is well-known that the Artin-Mazur dynamical zeta function of a
  hyperbolic or quasi-hyperbolic toral automorphism is a rational
  function, which can be calculated in terms of the eigenvalues of the
  corresponding integer matrix.  We give an elementary proof of this
  fact that extends to the case of general toral endomorphisms without
  change. The result is a closed formula that can be calculated by
  integer arithmetic only. We also address the functional equation and
  the relation between the Artin-Mazur and Lefschetz zeta functions.
\end{abstract}

\maketitle

\section{Introduction}

Any $d$-dimensional toral endomorphism is represented by an integer
matrix, $M \in \Mat(d,\ZZ)$, with action mod $1$ on the $d$-torus
$\TT^{d}\simeq\RR^{d}/\ZZ^{d}$; see \cite{AP} and \cite[Ch.~1.8]{KH}
for background and \cite[Ex.~1.16]{AA} for an illustration.  Important
aspects of the dynamical system $(\TT^{d},M)$ are related to its
periodic orbits and their distribution over $\TT^{d}$; compare
\cite{DEI,W}. The Artin-Mazur \cite{AM} dynamical zeta function
provides a generating function for the orbit counts that is
interesting both from an arithmetic and from a topological point of
view \cite{Fel,R}.  The latter was also Smale's approach \cite{Smale},
who related the Artin-Mazur and Lefschetz zeta functions of a
\emph{hyperbolic} toral automorphism and calculated both in terms of
eigenvalues.

In this note, we explain a different approach via elementary geometry
and linear algebra, which bypasses more advanced topological methods
as well as the need to calculate eigenvalues.  A key observation is
that all arguments apply to general toral endomorphisms without
additional effort.  We also treat the connection between the
Artin-Mazur and the Lefschetz zeta function and their functional
equations.  As we have learned along the way, most arguments we use
appear already in the literature, notably in \cite{Fel}, but at least
their combination seems to be new. Also, we make several steps
explicit to facilitate their computational use.

For $M\in \Mat (d,\ZZ)$ and $m\geq 1$, let $a_m$ be the number of
\emph{isolated} fixed points in $\TT^d$ of the $m$-th iterate $M^m$.
The starting point of our considerations is the identity
\begin{equation} \label{fixcount}
     a^{}_{m} = \lvert \ts \det(\one - M^m) \rvert \ts .
\end{equation}
This formula is well-known \cite{W,BHP} when no eigenvalue of $M^m$ is
$1$; it then follows from counting the number of points of $\ZZ^{d}$
in a fundamental domain of the lattice $(\one - M^{m})\ts
\ZZ^d$. Otherwise, \eqref{fixcount} is true because both sides are
zero.  Indeed, since the fixed points of $M^m$ form a closed subgroup
of $\TT^{d}$, they are either all isolated, or they form entire
subtori of positive dimension; see \cite[Appendix]{BHP} for a detailed
discussion of the subtorus case. Incidentally, when no eigenvalue of
$M^m$ is $1$, $a^{}_{m}$ is also the Reidemeister number of a toral
endomorphism, see \cite[Thm.~22 and p.~33]{Fel}, while, in general,
$a^{}_m$ is its Nielsen number.

Following \cite {AM}, the Artin-Mazur zeta function of a general $M
\in \Mat(d,\ZZ)$ is defined as
\begin{equation}\label{zeta-1} 
   \zeta^{}_{M} (z) := \ts \exp 
   \Bigl(\sum_{m=1}^{\infty} \frac{a^{}_{m}}{m} z^m \Bigr) =
   \prod_{m=1}^{\infty} (1-z^m)^{-c_m} .
\end{equation}
Here, the exponents $c^{}_m$ of the Euler product representation are
well-defined integers, see Proposition \ref{prop:euler} below.  An
explicit representation of $\zeta^{}_{M} (z)$ as a rational function
is given below in Theorem~\ref{thm:main}.  We stress that at least for
hyperbolic or quasihyperbolic toral endomorphisms, this result is
well-known by \cite{Smale} or \cite{Fel}; our focus is the elementary
method.

A matrix $M$ is called \emph{hyperbolic} when it has no eigenvalue on
the unit circle $\sss^{1}$.  Such toral automorphisms are expansive
\cite[p.~143]{W}.  Note that $M$ may possess eigenvalues on $\sss^{1}$
other than roots of unity (for instance, if one eigenvalue of $M$ is a
Salem number; see \cite{Wad,BRtorus} for examples). Integer matrices
without roots of unity in their spectrum constitute the
\emph{quasihyperbolic} cases, compare \cite{Wad} and references
therein, where formula \eqref{fixcount} still counts all fixed points.
For quasihyperbolic matrices $M$, the exponents $c_m$ are the cycle
numbers, which are related to the fixed point counts via
\begin{equation}\label{mobius}
    a^{}_m = \sum_{\ell | m} \ell\ts c^{}_{\ell}
    \quad \text{and} \quad
    c^{}_{m} = \frac{1}{m} \sum_{\ell | m} 
    \mu\bigl(\frac{m}{\ell}\bigr)\, a^{}_{\ell} \, .
\end{equation}
This follows from a standard application of M{\"o}bius inversion;
compare \cite{PW,BRW}. 

When roots of unity are among the eigenvalues of $M$, the Euler
product still exists, with the same relation between the counts of
(isolated) fixed points and the exponents $c^{}_m$, though the latter
can now be negative. Let us briefly illustrate this phenomenon in one
dimension.  Endomorphisms of $\ts\TT^{1}\!\simeq \sss^{1}$ are
represented by multiplication (mod $1$) with an integer $n$.  The
dynamical zeta function reads $\zeta^{}_{0}(z)=1/(1-z)$ and
\begin{equation} \label{dim-one}
   \zeta^{}_{n} (z) = 
    \frac{1-\sgn(n)\ts z}{1-\lvert n \rvert \ts z} \ts 
\end{equation}
for $n\neq 0$, due to our Theorem~\ref{thm:main} below (or a simple
direct calculation).  For $n=-1$, we get $\zeta^{}_{-1} (z) =
(1-z^2)/(1-z)^2$, thus $c^{}_{1} = 2$ and $c^{}_{2}=-1$, while
$c_{m}=0$ for all $m\ge 3$.  The negative $c^{}_2$ corresponds to the
fact that the two isolated fixed points of the map fail to be
\emph{isolated} for any even iterate, which is the identity.

\smallskip 
Finally, let us note that our arguments extend to the case of
nilmanifolds $X=G/\varGamma$ considered in \cite[Sec.~2.6]{Fel}, where
$G$ is a simply connected nilpotent Lie group and $\varGamma$ a
discrete subgroup such that $X$ is compact.  Namely, any endomorphism
$\varphi$ of $\varGamma$ extends to an endomorphism $\tilde\varphi$ of
$X$, the isolated fixed points of which are counted by $\big\lvert
\det \bigl(\one-\Lie (\tilde{\varphi}) \bigr)\big \rvert$, analogously
to \eqref{fixcount}.

\section{A related zeta function}

Let us start with the numbers $\widetilde{a}_{m} := \det(\one - M^m)$,
which can be viewed as signed fixed point counts, and the
corresponding zeta function
\begin{equation} \label{til-series}
   \widetilde{\zeta}^{}_{M} (z) = \exp \Bigl( \sum_{m = 1}^{\infty}
   \frac{\widetilde{a}_{m}}{m} z^m \Bigr).
\end{equation}
In Section~\ref{sec:lef}, we will see that this is actually a
Lefschetz zeta function, see Eq.~\eqref{rel-equals-lef} below.

For $A\in\Mat(d,\RR)$, let $\ep^k(A)$ be the induced linear map on the
exterior power $\ep^{k} (\RR^d)$.  In terms of the standard basis of
that space, $\ep^k(A)$ is represented by the matrix of all minors of
$A$ of order $k$; see \cite[Ch.~1.4]{Gant} for details.  This is an
integer matrix of dimension $\binom{d}{k}$, with $\ep^0(A) = 1$,
$\ep^1(A) = A$ and $\ep^d(A) = \det (A)$.

\begin{prop}
\label{pr-zeta-torus}
   For $M\in \Mat (d,\ZZ)$, we have 
   $\; \widetilde{\zeta}^{}_{M} (z) = \prod_{k=0}^{d}
    \det \bigl( \one - z \ep^k (M) \bigr)^{(-1)^{k+1}}$.
\end{prop}

Since all $\ep^k(M)$ are integer matrices, $\widetilde{\zeta}^{}_{M}
(z)$ is a rational function with numerator and denominator in
$\ZZ[z]$.  It can be calculated by integer arithmetic alone (many
algebraic program packages have the matrices of minors of arbitrary
order $k$ as built-in functions). Also, since $0$ is never a root of
the denominator, the series \eqref{til-series} for
$\widetilde{\zeta}^{}_{M}$ converges uniformly on sufficiently small
disks around $0$.

\begin{proof}[Proof of Proposition $\ref{pr-zeta-torus}$]
  This is analogous to \cite[Lemma~17]{Fel}: The assertion is
  immediate from the well-known formula in linear algebra
\begin{equation}
\label{trace-formula-torus}
   \det(\one -A)=\sum_{k=0}^{d} (-1)^{k}\trace(\ep^{k}(A))
\end{equation}
together with the power series identity
\begin{equation}
\label{trace-identity}
   \exp \Bigl(\sum_{m=1}^{\infty} \frac{\trace (A^{m})}{m}
   \,  z^m \Bigr) = \frac{1}{\det (\one - z A)} \ts ,
\end{equation}
which is omnipresent in connection with zeta functions of any kind (in
particular, it appears in the calculation of dynamical zeta functions
of shifts of finite type, see \cite{BL,R}).  We recall that
\eqref{trace-formula-torus} is proved by evaluating the characteristic
polynomial of $M$ at $1$, while \eqref{trace-identity} is a simple
consequence of the relation $\det \bigl( \exp (C)\bigr) = \exp\bigl(
\trace (C)\bigr)$ for square matrices $C$, together with the Taylor
series for $-\log(1-z)$, which is the case $d=1$ of
\eqref{trace-identity}.
\end{proof}

\smallskip
Let us also note that, in terms of the $d$ eigenvalues
$\lambda^{}_{1}, \ldots ,\lambda^{}_{d}\ts $ of $M$, one has the
relation $ \det \bigl( \one - z \ep^k (M) \bigr) = P^{}_{k} (z)$ with
the polynomials $P^{}_{0} (z) = 1-z$ and
\[
    P^{}_{k} (z) = \prod_{1\le \ell^{}_{1} < \ell^{}_{2} < 
   \cdots < \ell^{}_{k} \le d} (1 - z \ts \lambda^{}_{\ell^{}_{1}}\!
   \cdot\ldots\cdot \lambda^{}_{\ell^{}_{k}} )
\]
for $1\le k\le d$. 
This version is useful for the derivation of the functional
equation of $\widetilde{\zeta}^{}_{M}$.
\begin{lemma} \label{functional}
    If $M\in\Mat (d,\ZZ)$ with $D:= \det (M) \ne 0$, one has
    $\, \widetilde{\zeta}^{}_{M} (1/Dz) = B \ts \bigl( 
    \widetilde{\zeta}^{}_{M} (z) \bigr)^{(-1)^{d}}$, where
    $B=D$ for $d=1$ and $B=1$ otherwise.
\end{lemma}
\begin{proof}
First, a direct calculation shows that
\[
    P^{}_{k} \Bigl( \frac{1}{Dz} \Bigr) = \frac{1}{\beta^{}_{d-k}}
    \Bigl( \frac{-1}{z} \Bigr)^{\binom{d}{k}} P^{}_{d-k} (z)\ts ,
\]
where $\beta^{}_{0} = 1$ and $\beta^{}_{k} = D^{\binom{d-1}{k-1}}$ for
$1\le k\le d$. Note that each prefactor $\beta^{}_{k}$ involves
products of eigenvalues, but is symmetric in them and thus simplifies
to a power of the determinant.

Next, recall the binomial formula $\sum_{\ell=0}^{n} \binom{n}{\ell}
(-1)^{\ell} = \delta_{n,0}$ for $n\ge 0$, and insert the previous
polynomial identities into the product expression of
Proposition~\ref{pr-zeta-torus}.  Our claim follows, because the
prefactor that contains $z$ disappears by an application of the
binomial formula, while the prefactor with the determinants simplifies
to the factor $B$ by an analogous calculation; compare \cite[Lemma
19]{Fel} and its proof for a related argument.
\end{proof}
  
The special situation for $d=1$ is also immediate from
$\widetilde{\zeta}^{}_{n} (z) = (1-n z)/(1-z)$, as the
determinant is $n$; compare the example in the introduction.

\section{The Artin-Mazur zeta function}

To derive a formula for the dynamical zeta function, we observe that
$a_{m}= \widetilde{a}_{m} \sgn (\widetilde{a}_{m})$.  Hence the signs
of all nonzero $\widetilde{a}_{m}$ need to be determined. When $M$ is
quasihyperbolic, this is done in \cite[Lemma~2.1]{Wad}, see also the
proof of \cite[Lemma~15]{Fel}, but the argument works for general
$M\in\Mat(d,\ZZ)$, too: We employ the formula
\begin{equation}\label{eigen-1}
     \widetilde{a}_{m} = \det(\one-M^{m}) = 
     \prod_{j=1}^{d} (1 - \lambda_{j}^{m})
\end{equation}
with the $\lambda^{}_{j}$ as above. It is clear that neither complex
eigenvalues play a role (as they come in complex conjugate pairs, and
$(1 - \lambda^{m}) (1 - \bar{\lambda}^{m}) = \lvert 1 - \lambda^{m}
\rvert^2 \ge 0$), nor do eigenvalues $\lambda \in [-1,1]$ (because
then $1 - \lambda^m \ge 0$).  The remaining eigenvalues (evs) matter,
and one finds
\begin{equation}\label{signs}
    a_{m} \, = \, \widetilde{a}_{m} \,
    \bigl( (-1)^{\# \text{ real evs } < -1}\bigr)^{m}
    (-1)^{\# \text{ real evs outside } [-1,1] }
    \, =: \, \widetilde{a}_{m} \, \delta^{\ts m} \, \varepsilon .
\end{equation}

Inserting this into $\zeta^{}_{M} (z)$ and comparing with
$\widetilde{\zeta}^{}_{M} (z)$ gives
\begin{equation}\label{zeta-4}
    \zeta^{}_{M} (z) = \bigl(\ts\widetilde{\zeta}^{}_{M} 
    (\delta\ts z) \bigr)^{\varepsilon}
\end{equation}
for the dynamical zeta function of $M$.  Though \eqref{signs} involves
the eigenvalues of $M$, the signs $\delta$ and $\varepsilon$ can once
again be obtained by integer arithmetic alone.  When no eigenvalue of
$M$ is $\pm 1$, they are simply given by
\begin{equation}\label{signs-1}
   \delta = \sgn \bigl( \det (\one + M) \bigr)
   \quad \text{and} \quad
   \varepsilon = \delta  \sgn \bigl( \det (\one - M)\bigr) .
\end{equation}
In general, the signs can be defined by the one-sided limits
\[
   \delta = \lim_{\alpha\searrow \ts 0}\sgn 
   \bigl( \det ( (1+\alpha)\one + M ) \bigr)
   \quad \text{and} \quad
   \varepsilon = \delta \lim_{\alpha\searrow \ts 0} \sgn 
   \bigl( \det ((1+\alpha)\one - M)\bigr) ,
\]
which can be evaluated explicitly as follows.  Factorise $\det (x\one
- M) = (x-1)^{\sigma} (x+1)^{\tau} Q(x)$ with $Q\in\ZZ[x]$ and $Q(\pm
1)\neq 0$, where the non-negative integers $\sigma,\tau$ are
unique. This implies $\det (x\one +M) = (x-1)^{\tau} R(x)$ with $R\in
\ZZ[x]$ and $R(1)\ne 0$.  Consequently, one has
\begin{equation}\label{signs-2}
   \delta = \sgn \biggl( \frac{\det (x\one + M)}
   {(x-1)^{\tau}} \bigg|_{x=1} \biggr)
   \quad \text{and} \quad
   \varepsilon = \delta \sgn \biggl( \frac{\det (x \one - M)}
   {(x-1)^{\sigma}} \bigg|_{x=1} \biggr),
\end{equation}
which is used in our sample program in the appendix.

Let us summarise the result of our derivation so far.
\begin{theorem} \label{thm:main} 
  Consider a general toral endomorphism, represented by a matrix
  $M\in\Mat(d,\ZZ)$.  The associated Artin-Mazur zeta function,
  defined in terms of isolated fixed points, satisfies
\[
    \zeta^{}_M(z)=\prod_{k=0}^d\det \bigl(\one - \delta 
    z\ep^k(M)\bigr)^{\varepsilon\ts (-1)^{k+1}},
\]
  where the signs $\delta$ and $\varepsilon$ are given by
  Eq.~\eqref{signs-1} when $\pm 1$ is not an eigenvalue of $M$, and by
  Eq.~\eqref{signs-2} in general.  In particular, $\zeta^{}_{M} (z)$ is
  a rational function.  When no eigenvalue of $M$ is a root of unity,
  all fixed points are covered this way.  \qed
\end{theorem}

For quasihyperbolic toral endomorphisms, this result follows from an
analogous result for nilmanifolds \cite[Thm.~45]{Fel}, which is proved
by Reidemeister-Nielsen fixed point theory, while the case of
hyperbolic toral automorphisms is already treated in
\cite[Prop.~4.5]{Smale} in a slightly different formulation. Theorem
\ref{thm:main} covers the special cases of automorphisms for $d=2$
from \cite{KH,DEI,BRW}.  Let us note that, in the hyperbolic case, the
rationality of $\zeta^{}_{M}$ can also be seen as a consequence of the
general rationality result \cite{Man} proved by Markov partitions. Of
related interest is the approach of \cite{Miles}, which connects the
problem to an interesting class of $\ZZ^d$-actions.

Since $\delta$ and $\varepsilon$ are signs, in particular $\delta
=1/\delta$, the functional equation for $\zeta^{}_{M}$ is now
immediate from Lemma~\ref{functional} and Theorem~\ref{thm:main}.
\begin{corollary}
    When $D:= \det (M) \ne 0$, one has $\, \zeta^{}_{M} (1/Dz) =
    B^{\varepsilon}_{} \ts \bigl( \zeta^{}_{M} (z) \bigr)^{(-1)^{d}}$
    with $\varepsilon$ from Eq.~\eqref{signs-2},
    where $B=D$ for $d=1$ and\/ $B=1$ otherwise.  \qed
\end{corollary}

Let us also mention that
\begin{equation}\label{genfun}
   \sum_{m=1}^{\infty} a_{m} \ts z^m =
   \frac{z \, \zeta^{\ts\prime}_{M} (z)}{\zeta^{}_{M} (z)}
\end{equation}
is the ordinary power series generating function of the sequence
$(a_{m})^{}_{m\in\NN}$, which is still a rational function. Its radius
of convergence $\varrho^{}_{M}$ is always positive (it is the absolute
value of the smallest root of the denominator of \eqref{genfun} in
reduced form). Thus, when $\lim_{m\to\infty} \frac{a_{m+1}}{a_{m}}$
exists, $1/\varrho^{}_{M}$ is the asymptotic growth rate of the fixed
point counts, which provides a simple alternative to the approach in
\cite{Wad}. The limit exists precisely for hyperbolic (and hence
expansive) endomorphisms, as follows from \cite[Thm.~6.3]{CEW}. The
ratio as a growth measure is also employed in \cite[Thm.~16]{Lehmer},
where the case of unimodular roots is briefly discussed, too.

\smallskip 
Finally, we observe that the Artin-Mazur zeta function of a general
toral endomorphism can be written as an Euler product.

\begin{prop}\label{prop:euler}
   For any matrix $M\in\Mat(d,\ZZ)$, the associated Artin-Mazur
   zeta function $\zeta^{}_M(z)$ has an Euler product representation
   \eqref{zeta-1} with uniquely determined integers $c_n$. 
\end{prop}

\begin{proof}
  The Euler product representation \eqref{zeta-1} is equivalent to the
  relations \eqref{mobius}. These define rational numbers $c_n$ which
  we must show to be integers.  When $M$ is quasihyperbolic, this is
  true by their geometric interpretation as cycle numbers.  The
  general case (including $d=1$, which also follows from
  \eqref{dim-one}) can be proved by a deformation
  argument as follows.

  Fix $n$ and recall that $c_n$ is linear in the $a_k$, while
  $a_k=\delta^k \ts \varepsilon\det(\one-M^k)$ from \eqref{signs} with
  the signs $\delta$ and $\varepsilon$.  Hence, $c_n$ is a polynomial
  with rational coefficients in the entries of $M$, which can be
  written as $c_n=\varepsilon \ts P(M)$.  Here, $P$ itself depends on
  $\delta$, but neither on $\varepsilon$ nor on $M$.  For another
  matrix $M^{\ts\prime}$, let $\delta^{\ts\prime}$,
  $\varepsilon^{\ts\prime}$ and $c_{n}^{\ts\prime}$ denote the
  associated signs and numbers.  When $\delta^{\ts\prime}=\delta$, we
  thus have $c^{\ts\prime}_{n}=\varepsilon^{\ts\prime}\ts
  P(M^{\ts\prime})$.

  Let $\nu > 1$ be a common denominator of all coefficients of $P$
  and define the diagonal $d\!\times\!d$-matrix
  $N=\diag(2\delta,3,4,\ldots)$, which shares the sign $\delta$
  with $M$.  Consider now $M^{\ts\prime} = M+\nu^{r} N= \nu^{r}
  (N+\nu^{-r}M)$.  Since $N$ has distinct real eigenvalues, the
  eigenvalues of $N+\nu^{-r}M$, for sufficiently large $r$, are real
  and close to those of $N$. Then, $M^{\ts\prime}$ is hyperbolic with
  $\delta^{\ts\prime} = \delta$ by construction; in particular,
  $c^{\ts\prime}_{n}=\varepsilon^{\ts\prime} P(M^{\ts\prime})$ is
  integral. Since the difference $P(M)-P(M^{\ts\prime})$ is integral
  as soon as $r\geq 1$, it follows that $c_n=\varepsilon \ts P(M)$ is
  integral, too.
\end{proof}

\section{Interpretation as Lefschetz zeta function}
\label{sec:lef}

Suppose that $X$ is a compact differentiable manifold, assumed
orientable for simplicity, and $f:\, X\longrightarrow X$ is some
differentiable map.  In this situation (and also more generally),
there is a fixed point index $I_f\in\ZZ$ which satisfies the Lefschetz
trace formula
\begin{equation}
\label{lefschetz-trace-formula}
    I_f=\sum_k(-1)^k\trace\bigl(f_*|H_k(X;\QQ)\bigr),
\end{equation}
see \cite[Prop.~VII.6.6]{Dold} or \cite[Thm.~12.9]{Bredon}. Here,
$H_k(X;\QQ)$ denotes singular homology with coefficients in $\QQ$.  It
is a finite-dimensional $\QQ$-vector space in our situation, on which
$f$ acts by functoriality.

When all fixed points of $f$ are isolated, we have
$I_f=\sum_{x\in\fix(f)}i_f(x)$, where $i_f(x)\in\ZZ$ is the local
index of $f$ at $x$. If $x$ is a regular fixed point, meaning that $1$
is not an eigenvalue of the tangential map $T_x(f)$, the local index
is given by
\[
   i_f(x)=\sgn\bigl(\det(\one-T_x(f))\bigr) \, \in \, \{\pm 1\}.
\]

The Lefschetz zeta function associated to $f$ can be defined as
\[
   \zeta^L_f(z)=\exp\Bigl(\sum_{n\geq 1}\frac{z^n}{n}\, I_{f^n}\Bigr).
\]
This definition seems to appear first in \cite{Smale}; see also
\cite{Fel}.  By using the identity \eqref{trace-identity}, the trace
formula \eqref{lefschetz-trace-formula} applied to all iterates of $f$
implies that $\zeta^L_f (z)$ is a rational function,
\begin{equation}
\label{lefschetz-zeta-formula}
    \zeta^L_f (z)=\prod_{k}
   \det(\one-zf_*|H_{k}(X;\QQ))^{(-1)^{k+1}}.
\end{equation}

Let us now assume that $X=\TT^d=\RR^d/\ZZ^d$ as above, and that $f$ is
given by an arbitrary $M\in\Mat(d,\ZZ)$.  In this case, with the zeta
function $\widetilde{\zeta}^{}_{M}$ of
Proposition~\ref{pr-zeta-torus}, we have
\begin{equation} \label{rel-equals-lef}
  \widetilde\zeta^{}_{M}(z)=\zeta^L_{f}(z)\ts ,
\end{equation}
including a correspondence of all related formulas (for the case of
hyperbolic toral automorphisms, this was noted in \cite[p.\ 86, lines
7 and 16]{Fel}). Let us sketch a possible line of argument.  First, it
is well-known that the K\"unneth formula \cite[Thm.~3.2]{Bredon} gives
an isomorphism
\begin{equation}
\label{homology-torus}
    H_{k}(X;\QQ)\cong\ep^{k}(\QQ^d)
\end{equation}
such that the action of $f$ on $H_k$ corresponds to $\ep^k(M)$.  This
identifies the right hand sides of \eqref{trace-formula-torus} and
\eqref{lefschetz-trace-formula}, and similarly for
Proposition~\ref{pr-zeta-torus} and
Eq.~\eqref{lefschetz-zeta-formula}.  It follows that the left hand
sides of the corresponding pairs of equations are equal as well, that
is
\begin{equation}
\label{fixpoint-index-torus}
    I_f=\det(\one-M) , 
\end{equation}
and similarly $I_{f^n}=\det(\one-M^n)$; this appears also in
\cite[Prop.~4.15]{Smale} and in \cite{BBPT}.

A direct proof of \eqref{fixpoint-index-torus} without using the
Lefschetz trace formula can be done as follows.  Assume first that $1$
is not an eigenvalue of $M$. All fixed points $x$ of $f$ are then
regular, with the same local index
$i_f(x)=\sgn\bigl(\det(\one-M)\bigr)$, because the tangent space
$T_x(X)$ can be identified with $\RR^d$, where the action of $f$ is
given by $M$. Thus \eqref{fixpoint-index-torus} is immediate.  When no
eigenvalue of $M$ is a root of unity, the same applies to all iterates
of $f$.

For arbitrary $M$, we may use the following lemma.
\begin{lemma} \label{le-index} 
  Assume $X$ to be a compact Lie group of dimension $d$ and $f \! : \,
  X\to X$ a differentiable map.  Let $g(x)=x\cdot f(x)^{-1}$.  Then,
  $g$ acts on the $1$-dimensional\/ $\QQ$-vector space $H_d(X;\QQ)$ 
  by the scalar $I_f$.
\end{lemma}

Granting the lemma, \eqref{fixpoint-index-torus} follows easily
because, in the torus case, $g$ is given by $\one-M$, and we have
$H_d(X;\QQ)\cong\ep^d(\QQ^d)$.  We only sketch the proof of Lemma
\ref{le-index} and leave the details to the reader.  The fixed point
index can be defined as the homology intersection product in
$X\!\times\! X$ of the graph of $f$ and the diagonal,
$I_f=[\Gamma_X]\cdot[\Delta]$.  Since the automorphism of $X\times X$
given by $(x,y)\mapsto(x,x\cdot y^{-1})$ acts on the orientation by
$(-1)^{d}$, we get $I_f=(-1)^d[\Gamma_g]\cdot[X\times\{1\}]$.  The
assertion follows by a straightforward computation based on
decomposing $[\Gamma_g]$ according to the K\"unneth formula for
$X\!\times\! X$.

\bigskip
\bigskip
\section*{Appendix:\ A sample program for calculating $\zeta^{}_{M}$}

One can implement the explicit zeta function formulas of
Proposition~\ref{pr-zeta-torus} and Theorem~\ref{thm:main} in a simple
\textsc{Mathematica}${}^{\!\mbox{\tiny\textregistered}}$ program as
follows.

\smallskip 
\indent\indent\textsf{Clear[tilzeta,zeta,ord,sig,tau,del,eps,dim,one];} \\
\indent\indent\textsf{tilzeta[mat\_]:=(dim=Length[mat];
             Factor[Product[Det[ }\\
\indent\indent\indent \textsf{IdentityMatrix[Binomial[dim,k]] $-$ z 
             Minors[mat,k]]\^{}{((-1)\^{}(k+1))},$\{$k,0,dim$\}$]]);}\\
\indent\indent\textsf{ord[pol\_\ts, x\_]:=(tmp=pol; i=0;
    While[(tmp /. z$\ts\to\ts$x)==0, (i ++; tmp=D[tmp,z])]; i);}\\
\indent\indent\textsf{zeta[mat\_]:=(one=IdentityMatrix[Length[mat]];
             pol=Det[z one $-$ mat];}\\
\indent\indent\indent \textsf{sig=ord[pol,1]; tau=ord[pol, $-$ 1];}\\
\indent\indent\indent \textsf{del=Sign[Factor[Det[z one + 
             mat]/(z$-$1)\^{}tau] /. z$\ts\to\ts$1];}\\
\indent\indent\indent \textsf{eps=del Sign[Factor[pol/(z$-$1)\^{}sig] /. 
             z$\ts\to\ts$1];}\\
\indent\indent\indent \textsf{Factor[(tilzeta[mat] /. 
             z$\ts\to\ts\ts$del z)\^{}eps]);}

\smallskip\noindent 
The input is an integer matrix, in the standard format of a double
list.  The calculation is exact and reasonably fast for small
dimensions, and can be used up to dimension $8$ or $10$ say.

\bigskip
\bigskip
\section*{Acknowledgements}

It is a pleasure to thank Alex Fel'shtyn and Tom Ward for various
helpful suggestions and Doug Lind, John A.G.\ Roberts and Rudolf
Scharlau for discussions.  This work was supported by the German
Research Council (DFG), within the CRC 701.

\end{document}